\documentclass[10pt]{article}
\font\smallit=cmti10
\font\smalltt=cmtt10

\usepackage{amsthm,amsfonts,amssymb,amsmath,epsf, verbatim}
\usepackage{amssymb,latexsym,amsmath,epsfig,amsthm} 

\makeatletter

\renewcommand\section{\@startsection {section}{1}{\z@}
{-30pt \@plus -1ex \@minus -.2ex}
{2.3ex \@plus.2ex}
{\normalfont\normalsize\bfseries\boldmath}}

\renewcommand\subsection{\@startsection{subsection}{2}{\z@}
{-3.25ex\@plus -1ex \@minus -.2ex}
{1.5ex \@plus .2ex}
{\normalfont\normalsize\bfseries\boldmath}}

\renewcommand{\@seccntformat}[1]{\csname the#1\endcsname. }

\makeatother

\newtheorem{theorem}{Theorem}
\newtheorem{lemma}[theorem]{Lemma}
\newtheorem{corollary}[theorem]{Corollary}
\newtheorem{proposition}[theorem]{Proposition}

\newtheorem{conjecture}[theorem]{Conjecture}

\theoremstyle{definition}

\begin{document}

\begin{center}
\uppercase{\bf On the small prime factors of a non-deficient number}
\vskip 20pt

{\bf Joshua Zelinsky}\\
{\smallit Department of Mathematics, Hopkins School, New Haven, Connecticut, USA}\\
{\tt zelinsky@gmail.com}\\ 

\end{center}

\vskip 20pt
\centerline{\smallit Received: , Revised: , Accepted: , Published: } 
\vskip 30pt

\centerline{\bf Abstract}
\noindent
Let $\sigma(n)$ be the sum of the positive divisors of $n$. A number $n$ is non-deficient if $\sigma(n) \geq 2n$. We establish new lower bounds for the number of distinct prime factors of an odd non-deficient number in terms of its second smallest, third smallest and fourth smallest prime factors. We also obtain tighter bounds for odd perfect numbers.  We also discuss the behavior of $\sigma(n!+1)$, $\sigma(2^n+1)$, and related sequences.

\pagestyle{myheadings}
\markright{\smalltt INTEGERS: 21 (2021)\hfill}
\thispagestyle{empty}
\baselineskip=12.875pt
\vskip 30pt

\section{Introduction}

We will write $\sigma(n)$ to be the sum of the positive divisors of $n$. Recall, that $\sigma(n)$ is a multiplicative function. Recall further that numbers where $\sigma(n)=2n$ are said to be perfect, numbers where $\sigma(n) < 2n$ are said to be deficient, and numbers where $\sigma(n) > 2n $ are said to be abundant. A large amount of prior work has been done on estimating the density of abundant numbers. See for example \cite{Klyve} and \cite{Kobayashi}.

We will throughout this paper write $n = p_1^{a_1}p_2^{a_2} \cdots p_k^{a_k}$ where $p_1,p_2\cdots p_k$ are primes satisfying $p_1 < p_2 \cdots < p_k$. Bounds giving a lower bound for $n$ in terms of $p_i$ for either small $i$ or $i$ close to $k$ exist already in the literature, under the assumption that $n$ is non-deficient or under the stronger assumption that $n$ is an odd perfect number. 

Acquaah and Konyagin showed \cite{AcquaahKonyagin} that if $N$ is an odd perfect number, then \begin{equation}p_k < (3N)^{\frac{1}{3}},\label{AcquaahKonyagin1}\end{equation}and \begin{equation}p_{k-1} < (2N)^{\frac{1}{4}}.\end{equation} Building on their work, the author
\cite{Zelinskythird} showed that \begin{equation}p_{k-1} < (2N)^{\frac{1}{5}},\end{equation} and that \begin{equation} p_{k-1}p_k < 6^{\frac{1}{4}}N^{\frac{1}{2}}.\label{Zelinskydouble}
\end{equation}

However, Equations \ref{AcquaahKonyagin1} through \ref{Zelinskydouble} give a lower bound for $N$ in terms of $p_k$ and $p_{k-1}$. They do  not give a lower bound for $k$, except in so far as we have upper bounds for $N$ in terms of $k$, and the inequalities can then be chained together. In contrast, for $p_i$, when $i$ is small, the natural rout seems to go mainly in the other direction, using $p_i$ to bound $k$ from below, and then use those results to obtain a lower bound on $N$.

The main results of this paper are improved lower bounds on $k$ in terms of $p_2$ and $p_3$ when $n$ is odd and not deficient. We prove similar bounds with $p_4$ when $n$ is odd and perfect, and a slightly weaker result when $n$ is a primitive non-deficient number. Using the same techniques we will also produce estimates for $\sigma(m)$ when $m$ is near but not equal to $n!$.

Kishore \cite{Kishore} proved that if $n$ is an odd perfect number with $k$ distinct prime factors, and $2 \leq i \leq 6$, then one must have $$p_i < 2^{2^{i-1}}\left(k-i+1\right).$$ Note that these bounds are all linear in $p_i$ and $k$.  In this paper, for $i$ in the range $2 \leq i \leq 4$, we improve on Kishore's results and give better than linear bounds. Moreover, for $i =2$ and $i=3$, our results apply not just to odd perfect numbers but also to odd abundant numbers. 

For $p_1$, better than linear bounds are known, due to \cite{Norton} and the author \cite{Zelinskybig}. 
Most of the results in this paper rely on the same fundamental techniques as in \cite{Norton} and \cite{Zelinskybig}. The first section of this paper extends those techniques to these small prime divisors. The second section uses the same techniques to look at a related problem involving the sum of the divisors of numbers near $n!$.

We will write $h(n) = \frac{\sigma(n)}{n}$. Note that $h(n)$ is frequently called the {\emph{abundancy}} of $n$, but it has also been referred in the literature  as the {\emph{abundancy index}} of $n$ as well as just the {\emph{index}} of $n$. A number is perfect if and only if $h(n) =2$, is abundant if and only $h(n)>2$, and is deficient if and only $h(n) <2$.

One has that \begin{equation}h(n) = \frac{1}{n}\sum_{d|n} d= \sum_{d|n}\frac{1}{d}\end{equation} from which it follows that \begin{equation}\label{h inequality} h(mn)> h(n)\end{equation} whenever $m>1$. A consequence of \ref{h inequality} is that any multiple of an abundant number is abundant, and any multiple of a perfect number other than the number itself is abundant. Motivated by this consequence, there is a notion of  a {\emph{primitive abundant number}} which is a number $n$ which is abundant but where proper all divisors of $n$ are not abundant. There is some inconsistency in the literature here; some authors define primitive abundant numbers as including perfect numbers where all proper divisors are  deficient. In some respects, this is a more natural set to investigate. Others term this set to be the set of {\emph{primitive non-deficient numbers}}, which is the term we will use. The first substantial work on this topic was done by Dickson who proved that given a fixed $k$, there are only a finite number of primitive non-deficient  numbers with $k$ distinct prime factors\cite{Dickson}. Prior work on these sets looked at their density. See especially \cite{Avidon} as well as \cite{Lichtman} and \cite{Schinzel}.

We will write $H(n) = \Pi_{i=1}^k \frac{p_i}{p_i-1}$. It is well-known and not hard to show that $h(n) \leq H(n)$ (and this inequality is strict as long as $n>1$). In fact, $$ \lim_{m \rightarrow \infty} h(n^m) = H(n).$$ Thus, $H(n)$ is the best possible upper bound on $h(n)$ when we know only the distinct prime divisors of $n$, but have no other information about $n$. 

We will write $P_j$ as the $j$th prime number (note the capital letter distinguishing this from the $p_i$). 

Given a prime $P_j$ and a real number $\alpha >1$, we will write 
$b_{\alpha}(P_j)$ to be the smallest positive integer $b$ such that 
$$\alpha < \left(\frac{P_j}{P_j-1}\right)\left(\frac{P_{j+1}}{P_{j+1}-1}\right) \cdots \left(\frac{P_{j+b-1}}{P_{j+b-1}-1}\right).$$

We will similarly define $a_{\alpha}(j)$ as $b_{\alpha}(P_j)$.

These two functions, $a_{\alpha}(j)$ and $b_{\alpha}(P_j)$, are generalizations of the functions labeled as $a(j)$ and $b(P_j)$
in \cite{Norton} and \cite{Zelinskybig}. In particular, $a(j) =a_{2}(j)$ and $b(j) = b_2(j)$. The use of $\alpha =2$ arises naturally in the context when one is interested in bounding $k$ in terms of the smallest prime divisor of $n$, but examining other prime divisors requires looking at other values of $\alpha$. Note that if $P_j$ is the smallest prime divisor of an odd perfect or odd abundant number $n$ then $k \geq a(j)$. 

Norton proved
\begin{theorem} Let $t$ be a positive integer. Then: \begin{equation}\label{Norton 1-1}a(t) >t^2 -2t - \frac{t+1}{\log t} - \frac{5}{4} - \frac{1}{2t} - \frac{1}{4t \log t}. \end{equation}
\end{theorem}

In fact, $a(t)$ is asymptotic to $\frac{1}{2}t^2\log t$. When Norton was writing, bounds on the prime number theorem and related functions were not strong enough to easily give an explicit lower bound on $a(t)$ which had the correct lead term. Subsequently, the author proved an explicit version with the correct asymptotic \cite{Zelinskybig}:
\begin{theorem}
 For all $t>1$ we have \begin{equation} a(t) \geq \frac{t^2}{2}\left(\log t - \frac{3}{2}\log \log t +\frac{1}{20} + \frac{\log \log t}{\log t} \right) - t +1   \label{Improved a(n) bound}.\end{equation}
 \end{theorem}

Notice that this theorem has the correct first lead term. 

Norton also proved that \begin{equation} \label{Norton 5} a(n) = \frac{1}{2}n^2 \log n + \frac{1}{2} n^2 \log \log n -\frac{3}{4}n^2 + \frac{n^2 \log \log n }{2 \log n} + O\left(\frac{n^2}{\log n}\right).\end{equation} 

Thus, Inequality \ref{Improved a(n) bound} is not as strong as it could be. Although Inequality \ref{Improved a(n) bound} has the correct lead term, the later terms do not match the asymptotic. 

In the final section of this paper we use the same techniques to give a new proof of  a theorem about the abundancy of numbers of the form $m! + z$ where $z$ is a fixed non-zero integer. We will prove the following.

\begin{theorem} Let $z$ be a non-zero integer. Then $$\lim_{m \rightarrow \infty} h(m! + z) = h(|z|). $$ \label{m! + z theorem}
\end{theorem}
Theorem \ref{m! + z theorem} appears to be a known folklore result; we are uncertain if it has been previously published. 

\section{Small prime factors of non-deficient numbers}

We require the following explicit version of Mertens' Theorem.

\begin{lemma} We have 
\begin{equation} e^\gamma (\log x)\left(1 - \frac{1}{2 \log^2 x}\right)  <  \prod_{p \leq x } \frac{p}{p-1} < e^\gamma (\log x)\left(1 + \frac{1}{2 \log^2 x}\right). 
\end{equation} Here $\gamma$ is Euler's constant. The lower bound is valid for $x>3$ and the upper bound is valid for $x \geq 286.$
\label{Rosser Schoenfeld bounds for Mertens product}
\end{lemma}
\begin{proof} This is proven in  \cite{RosserSchoenfeld1962}.
\end{proof}

The following asymptomtic is a straightforward consequence of this explicit version of Mertens' Theorem and the Prime Number Theorem.

\begin{theorem}\label{asymptotic for b_alpha} Let $\alpha >1$ be fixed. Then $$b_{\alpha}(p) \sim \frac{p^{\alpha}}{\alpha \log p},$$
and $$a_{\alpha}(n) \sim \frac{n^{\alpha} \log n}{\alpha}.$$
\end{theorem}

Much of our work will be trying to make Theorem \ref{asymptotic for b_alpha} explicit for certain specific values of $\alpha$. We will for our purposes only need to construct explicit lower bounds, but it may be interesting to construct similarly explicit upper bounds.

We will write $S(x)=\prod_{p \leq x}\frac{p}{p-1}$, and will write $$Q(x,y)= \prod_{y < p \leq x} \frac{p}{p-1}=\frac{S(x)}{S(y)}$$ and will when we write this assume that $x>y >3$. 

\begin{lemma}\label{obnoxious technical lemma} Let $A$, $B$ and $\alpha$ be real numbers.  Assume that $A > B > 6$, and that $3 \geq \alpha >1$. Assume further that 
$$A + \frac{1}{2A} \geq \alpha\left(B - \frac{1}{2B}\right).$$ Then 
$A \geq \alpha B - \frac{2}{B}$. 
\end{lemma}
\begin{proof} We note that $A+\frac{1}{A}$ is an increasing function in $A$ for all $A>1$. So we just need to verify that 
$\left(\alpha B - \frac{2}{ B}\right) + \frac{1}{2(\alpha B - \frac{2}{ B})} \leq  \alpha(B - \frac{1}{2B}).$ This is the same as verifying that 
$$\alpha B - \frac{\alpha}{2B} - \left(\left(\alpha B - \frac{2}{ B}\right) + \frac{1}{2( \alpha B - \frac{2}{ B})}  \right) \geq 0. $$ We can rewrite the expression on the right-hand side above as

$$ \frac{\left(-\alpha +4\right)(\alpha B + \frac{2}{B}) -B}{2B\left(\alpha B - \frac{2}{B}\right)}$$ which has a positive denominator (since $B > \sqrt{2}$), so we just need to look
at $$(-\alpha +4)(\alpha B + \frac{2}{B} - B). $$ Since $\alpha \leq 3$, we have $$(-\alpha +4)(\alpha B + \frac{2}{B} - B) \geq (\alpha B + \frac{2}{B}) - B > \frac{2}{B} >0.$$

\end{proof}

We will need a few explicit estimates of certain functions of primes.
We have from \cite{Dusart 1999} that \begin{equation}\label{Dusart 1999 pi bound} \frac{x}{\log x}\left(1+ \frac{0.992}{\log x}\right) \leq \pi(x) \leq \frac{x}{\log x}\left(1 + \frac{1.2762}{\log x} \right),
\end{equation} 
with the lower bound valid if $x \geq 599$ and the upper bound valid for all $x > 1$.

We also need from \cite{Zelinskybig} that for $n \geq 2$,

 \begin{equation}\label{Modified Dusart Pn bound}  P_n \geq n\left(\log n + \log \log n -1 + \frac{32}{31 (\log n)^2} \right). \end{equation}

\begin{theorem} Let $\alpha$ be a real number such $1< \alpha \leq 3$ and let $p$ be a prime with $p \geq 853$.  Then $$b_{\alpha}(p) \geq \frac{p^\alpha}{\alpha \log p} -  \frac{1.669p^\alpha}{\alpha \log^2  p}  - \frac{ 0.662 p^\alpha}{\alpha \log^3  p} - \frac{p}{\log p} - \frac{1.2762p}{\log^2p } +1  .$$\label{General b alpha bound}
\end{theorem}
\begin{proof}  Let $p \geq 601$ be a prime.  Let $q$ be the smallest prime such that $Q(p,q) \geq \alpha$. 

We have from Lemma \ref{Rosser Schoenfeld bounds for Mertens product} that \begin{equation}\alpha \leq  Q(q,p) \leq \frac{e^\gamma (\log q)\left(1 + \frac{1}{ 2\log^2 q}\right)}{e^\gamma (\log p)\left(1 - \frac{1}{2 \log^2 p}\right)}  =  \frac{ (\log q)\left(1 + \frac{1}{2 \log^2 q}\right)}{ (\log p)\left(1 - \frac{1}{2 \log^2 p}\right)}.  \end{equation}

If we set $A= \log q$, and $p = \log B$, we have
$$A+\frac{1}{2A} \geq \alpha B (1-\frac{2}{B^2}), $$ so we may apply Lemma \ref{obnoxious technical lemma} to get that
$\log q \geq \alpha \log p - \frac{2}{\log p}$. Then we have  
$q \geq p e^\frac{2}{\log p} $.

We have then $$b_{\alpha(p)} \geq \pi(p^{\alpha} e^\frac{-2}{\log p}) - \pi(p) +1.$$  We note that $e^{\frac{-2}{\log p}} \geq 1 - \frac{2}{\log p}$ and so
$$b_{\alpha(p)} \geq \pi\left(p^{\alpha}(1- \frac{2}{\log p})\right) - \pi(p) +1 . $$
Since $p \geq 853$, we have $p^{\alpha}\left(1- \frac{2}{\log p}\right) \geq 599$, and may apply Lemma \ref{Dusart 1999 pi bound} to  get that $$b_{\alpha(p)} \geq   \frac{p^{\alpha}(1- \frac{2}{\log p})}{\log p^{\alpha}(1- \frac{2}{\log p})}\left(1+ \frac{0.992}{\log p^{\alpha}(1- \frac{2}{\log p})}\right) -\frac{p}{\log p}\left(1 + \frac{1.2762}{\log p} \right) +1. $$ and so 
$$b_{\alpha(p)} \geq   \frac{p^{\alpha}(1- \frac{2}{\log p})}{ \alpha \log p}\left(1+ \frac{0.992}{\alpha \log p}\right) -\frac{p}{\log p}\left(1 + \frac{1.2762}{\log p} \right) +1. $$
Since $\alpha \leq 3$, we have 
$$b_{\alpha(p)} \geq  \frac{p^{\alpha}(1- \frac{2}{\log p})}{ \alpha \log p}\left(1+ \frac{0.331}{\log p}\right) -\frac{p}{\log p}\left(1 + \frac{1.2762}{\log p} \right) +1.$$ 

Multiplying out the terms gives us the desired inequality.
\end{proof}

We will need some specific versions of the above theorem for certain values of $\alpha$.

\begin{theorem} For all primes $p \geq 3$, one has 
\begin{equation}b_{\frac{4}{3}}(p) \geq \frac{p^{\frac{4}{3}}}{\frac{4}{3} \log p} -  \frac{1.25175
p^\frac{4}{3} }{\log^2  p}  - \frac{ 0.4965p^\frac{4}{3}}{ \log^3  p} - \frac{p}{\log p} - \frac{1.2762p}{\log^2p } +1\label{4/3s explicit b equation}, \end{equation}

\begin{equation}b_{\frac{16}{15}}(p) \geq \frac{p^{\frac{16}{15}}}{\frac{16}{15} \log p} -  \frac{1.5646875
p^\frac{16}{15} }{\log^2  p}  - \frac{0.620625p^\frac{16}{15}}{ \log^3  p} - \frac{p}{\log p} - \frac{1.2762p}{\log^2p } +1,\label{16/15s explicit b equation} \end{equation}

\begin{equation}b_{\frac{256}{255}}(p) \geq  \frac{p^{\frac{256}{255}}}{\frac{256}{255} \log p} -  \frac{1.662481
p^\frac{256}{255} }{\log^2  p}  - \frac{0.65942p^\frac{256}{255}}{ \log^3  p} - \frac{p}{\log p} - \frac{1.2762p}{\log^2p } +1\label{256/255 explicit b equation}, \end{equation}
and 
\begin{equation}b_{
\frac{26325}{26257}}(p) \geq  \frac{p^{
\frac{26325}{26257}}}{
\frac{26325}{26257} \log p} -  \frac{1.6647
p^ \frac{26325}{26257} }{\log^2  p}  - \frac{0.6603p^\frac{26325}{26257}}{ \log^3  p} - \frac{p}{\log p} - \frac{1.2762p}{\log^2p } +1\label{52514/26325 explicit b equation}. \end{equation}
\end{theorem}

\begin{proof} We will only discuss how to prove Equation \ref{4/3s explicit b equation}. The other equations have very similar proofs. Theorem \ref{General b alpha bound} implies that Equation \ref{4/3s explicit b equation} is true for all primes $p \geq 853$. We then use a computer program to check that it is true for all primes $p$ with $3 \leq p \leq 853$. The others are done similarly. 
\end{proof}

\begin{proposition} Suppose that $N$ is an odd number which is either perfect or abundant. Let $p_2$ be  second smallest prime divisor of $N$. Let $k$ be the number of distinct prime divisors of $N$. Then $k \geq b_{\frac{4}{3}}(p_2)  +1$.
\label{second smallest prime bound}
\end{proposition}
\begin{proof} We have $h(N) \geq 2$. We assume that $N$ has the factorization
$p_1^{a_1}p_2^{a_2} \cdots p_k^{a_k}$ with $p_1 < p_2 < p_3 \cdots < p_k$. We have then
$$2 \leq h(N)< H(n) = \frac{p_1}{p_1-1} \frac{p_2}{p_2-1} \cdots  \frac{p_k}{p_k-1}.$$
We have that in the most extreme case, $p_1=3$, in which case $\frac{p_1}{p_1-1} = \frac{3}{2}$. We thus have $$\frac{3}{2} \frac{p_2}{p_2-1} \cdots  \frac{p_k}{p_k-1} > 2,$$ and hence
 $$\frac{p_2}{p_2-1} \cdots  \frac{p_k}{p_k-1} > \frac{4}{3}.$$ We thus have $k-1 \geq b_{\frac{4}{3}}(p_2).$
\end{proof}

Applying Proposition \ref{second smallest prime bound} with Inequality \ref{4/3s explicit b equation} we obtain the next result.

\begin{corollary}
 Suppose that $N$ is an odd number which is either perfect or abundant. Let $p_2$ is second smallest prime divisor of $N$. Let $k$ be the number of distinct prime divisors of $N$. Then $$k \geq \frac{p_2^{\frac{4}{3}}}{\frac{4}{3} \log p_2} -  \frac{1.25175
p_2^\frac{4}{3} }{\log^2  p_2}  - \frac{ 0.4965p^\frac{4}{3}}{ \log^3  p_2} - \frac{p_2}{\log p_2} - \frac{1.2762p_2}{\log^2p_2 } +2.$$ 
\label{second smallest prime bound2}
\end{corollary}

The bound from Corollary  \ref{second smallest prime bound2} exceeds Kishore's bound when $p_2 \geq 139$. The bound also applies to all odd non-deficient numbers, whereas Kishore's bound applies only to odd perfect numbers. With a little computation we can use Corollary  \ref{second smallest prime bound2} to obtain the following bound which is stronger than Kishore's bound for any value of $p_2$.

\begin{corollary} If $N$ is an odd abundant or odd perfect number with $k$ distinct prime factors and second smallest prime factor $p_2$, then $p_2 \leq 3k -1$.
\label{newlinearorp2}
\end{corollary}

\begin{proposition}Suppose that $N$ is an odd number which is either perfect or abundant. Let $p_3$ be the third smallest prime divisor of $N$. Let $k$ be the number of distinct prime divisors of $N$. Then $k \geq b_{\frac{16}{15}}(p_3)  +3$.
\end{proposition}
\begin{proof} This proof is essentially the same as the proof of Proposition \ref{second smallest prime bound}. However, now the worst case scenario is that the smallest prime is 3 and the second smallest prime is 5. 
\end{proof}

\begin{corollary}
 Suppose that $N$ is an odd number which is either perfect or abundant. Let $p_2$ be the smallest prime divisor of $N$. Let $k$ be the number of distinct prime divisors of $N$. Then $$ b_{\frac{16}{15}}(p_3) \geq \frac{p_3^{\frac{16}{15}}}{\frac{16}{15} \log p_3} -  \frac{1.5646875
p_3^\frac{16}{15} }{\log^2  p_3}  - \frac{0.620625p_3^\frac{16}{15}}{ \log^3  p_3} - \frac{p_3}{\log p_3} - \frac{1.2762p_3}{\log^2p_3 } +3.$$ \label{p3 corollary}
\end{corollary}

Here the inequality is better than Kishore's inequality when $p \geq 282767$.

Due to the large value where the inequality in Corollary \ref{p3 corollary} becomes better than Kishore's inequality, it is not feasible to construct a nice linear inequality like Corollary \ref{newlinearorp2} for $p_3$. 

We now turn to smallest bounds on the fourth prime factor. Here, we must restrict ourselves finally to odd perfect numbers. This is because there are actually odd abundant numbers with three distinct prime factors; $N=945$ is an example.

We recall one of the oldest results about odd perfect number which is due to Euler.
\begin{lemma}\label{Euler form for OPN} (Euler) If $N$ is an odd perfect number then have $N= p^em^2$ for some prime $p$ where $(p,m)=1$ and $p \equiv e \equiv 1$ (mod 4).
\end{lemma}
In fact, Euler's result applies more generally to any odd number $n$ satisfying $\sigma(n) \equiv 2$ (mod 4).

We also note the following well known Lemma where for completeness we include a proof.
\begin{lemma} Let $N$ be an odd perfect number. Then $105 \not |N$.
\end{lemma}
\begin{proof} Let $N$ be an odd perfect number, and assume that $105|N$. Since $3 \equiv 3$ (mod 4), and $7 \equiv 3$ (mod 4), by Euler's result, we must have $3^2|N$ and $7^2|N$. But $(3^2)(5)(7^2)$ is an abundant number. 
\end{proof}

Using the same approach we can further restrict the prime divisors of an odd perfect number.

\begin{lemma} Let $N$ be an odd perfect number. Assume that $3|N$, $5|N$ and $11|N$. Then one must have $5||N$.
\label{OPN with 3, 5 and 11}
\end{lemma}
\begin{proof} Suppose that $N$ is an odd perfect number where $(3)(5)(11)|N$. Note that since $3 \equiv 11 \equiv 3$ (mod 4), we must from Euler's result, have $3^2|N$ and $11^2|N$. First, we will show that $3^4|N$. Assume that $3^4\not|N$. Since $3 \equiv 3$ (mod 4), we must have $3^2|N$. Thus, $\sigma(3^2)=13|N$. But $(3^2)5(11^2)(13)$ is abundant. We thus may assume that $3^4|N$. \\ 

Note that $\sigma(11^2) = (7)(19)$, so if $11^2||N$, we would have $105|N$ which by the prior Lemma is impossible. Thus, we conclude that $11^4|N$. However, $(3^4)(5^2)(11^4)$ is an abundant number. So we must have that $5|N$.
\end{proof}

\begin{lemma} Let $N$ be an odd perfect number. Suppose that $3|N$, $5|N$, and $13|N$. Then we must have either $5||N$ or $3^2||N$.
\label{OPN 3 5 13 case}
\end{lemma}
\begin{proof}Assume that $N$ is an odd perfect number and that $3|N$, $5|n$ and $13|N$.  We note that if $13$ were the special prime, then since $7|(13+1)$, we would have $7|N$, which is impossible since no odd perfect number is divisible by $105$.

Assume that $5^2|N$ and that $3^3|N$; the second implies that we must have $3^4|N$. But $(5^2)(3^4)(13^2)|N$ and then $(5^2)(3^4)(13^2)$ is abundant, which is a contradiction.

\end{proof}

\begin{lemma} Let $N$ be an odd perfect number, with $N=p_1^{a_1}p_2^{a_2}p_3^{a_3} \cdots p_k^{a_k}$ with $p_1 < p_2 < p_3 \cdots < p_k$. Then $h(p_1^{a_1}p_2^{a_2}p_3^{a_3}) < \frac{255}{128}$.
\end{lemma}
\begin{proof} Note that if $p_1 \geq 5$, then 
$H(p_1^{a_1}p_2^{a_2}p_3^{a_3}) \leq \frac{5}{4}\frac{7}{6}\frac{11}{10} < \frac{255}{128}$. Since $h(m) < H(m)$ for all $m>1$, this means that we would have $h(p_1^{a_1}p_2^{a_2}p_3^{a_3}) < \frac{255}{128}$.  So we may assume that $p_1=3$.

Similarly, since $\frac{3}{2}\frac{7}{6}\frac{11}{10} < \frac{255}{128}$ we may assume that $p_2=5$.  Since $\frac{3}{2}\frac{5}{4}\frac{17}{16} = \frac{255}{128}$, we may assume that $p_3 = 11$ or that $p_3=13$.

Let us first consider the situation where $p_3=11$. Then from Lemma \ref{OPN with 3, 5 and 11}, we must have $a_2=1$. Then $$h(p_1^{a_1}p_2^{a_2}p_3^{a_3}) = h(3^{a_1})h(5)h(11^{a_3}) < H(3)h(5)H(11) = \frac{3}{2}\frac{6}{5}\frac{11}{10} < \frac{255}{128}.$$

The situation where $p_3=13$ is nearly identical, but we use Lemma \ref{OPN 3 5 13 case} instead of Lemma \ref{OPN with 3, 5 and 11}.
\end{proof}

\begin{proposition} Suppose that $N$ is an odd perfect number, and let $p_4$ be the fourth smallest prime factor of $N$. Assume that $N$ has $k$ distinct prime factors. Then $$k \geq b_{\frac{256}{255}}(p_4) +3.$$
 \label{OPN P_4 bound}
\end{proposition}
\begin{proof} Let $N$ be an odd perfect number. Assume that $N$ has the factorization
$p_1^{a_1}p_2^{a_2} \cdots p_k^{a_k}$ with $p_1 < p_2 < p_3 \cdots < p_k$.  We must have that 
$$h(N)=2=h(p_1^{a_1}p_2^{a_2}p_3^{a_3})h(p_4^{a_4}\cdots p_k^{a_k}) < \frac{255}{128}H(p_4^{a_4}\cdots p_k^{a_k}). $$ 
We have then that $H(p_4^{a_4}\cdots p_k^{a_k}) \geq \frac{256}{255}$ The result then follows from the same sort of argument as before. \end{proof}

We then obtain a corresponding corollary as before.

\begin{corollary} Suppose that $N$ is an odd perfect number, and let $p_4$ be the fourth smallest prime factor of $N$. Assume that $N$ has $k$ distinct prime factors. Then $$k \geq    \frac{p^{\frac{256}{255}}}{\frac{256}{255} \log p} -  \frac{1.662481
p^\frac{256}{255} }{\log^2  p}  - \frac{0.65942p^\frac{256}{255}}{ \log^3  p} - \frac{p}{\log p} - \frac{1.2762p}{\log^2p } +4.$$
\end{corollary}

We can make a slightly weaker version of Proposition \ref{OPN P_4 bound} which does apply to some non-deficient numbers. Recall a number $M$ is a primitive non-deficient number if $M$ is not deficient and all divisors of $M$ are deficient.

\begin{lemma} Let $M$ be a primitive non-deficient number with at least four distinct prime factors. Suppose that $3|M$, $5|M$, and $7|M$. Then one has either $3||M$ or  $3^2||M$. Moreover, if $3^2||M$ then $5||M$ and $7||M$.\label{3-5-7-primitive} 
\end{lemma}
\begin{proof} First, note that $(3^3)(5)(7)$ is abundant. So one must have $3||M$ or $3^2||M$. 

Now assume that $3^2||M$. Then since $(3^2)(5^2)7$ and $(3^2)5(7^2)$ are both abundant, we must have $5||M$ and $7||M$.

\end{proof}

\begin{lemma}Let $M$ be a primitive non-deficient number with at least four distinct prime factors. Suppose that $3|M$, $5^2|M$, and $11|M$. Then either $3||N$ or $3^2||N$.\label{3-5-11 primitive lemma}
\end{lemma}
\begin{proof} This follows from noting that $(3^3)(5^2)(11)$ is abundant. 
\end{proof}

\begin{lemma} Let $M$ be a primitive non-deficient number with at least four distinct prime factors. Suppose that $3|M$, $5^2|M$, and $13|M$. Then $3^5 \not|M$.\\

Furthermore, under the same assumptions if $3^4|M$, then $5^2||M$ and $13^2 \not|M.$ \\

Under the same assumptions, if $3^3|M$, then either $5^3 \not| M$ or $13^2 \not |M$.

\label{3-5-13 primitive}
\end{lemma}
\begin{proof} This follows from noting that $(3^5)(5^2)(13)$, $(3^4)(5^3)(13)$, $(3^4)(5^2)(13^2)$,  $(3^3)(5^3)(13^2)$ are all abundant numbers. 
\end{proof}

\begin{theorem} Let $M$ be a primitive non-deficient number with at least four distinct prime factors. Assume further that $M=p_1^{a_1}p_2^{a_2}p_3^{a_3}p_4^{a_4} \cdots p_k^{a_k}$ where $p_1$, $p_2$, $p_3 \cdots p_k$ are primes and $p_1 < p_2 < p_3 <p_4 < \cdots p_k$. Then we have $h(p_1^{a_1}p_2^{a_2}p_3^{a_3}) \leq \frac{52514}{26325} = 1.9948 \cdots $. \label{h bound for prim non-deficient with at least four distinct prime factors}
\end{theorem}
\begin{proof} The method is essentially the same as our earlier approach. Set $M=p_1^{a_1}p_2^{a_2} \cdots p_k^{a_a}$.  If $p_1>3$, then $$h(p_1^{a_1}p_2^{a_2}p_3^{a_3}) < H((5)(7)(11)) = \frac{5}{4}\frac{7}{6}\frac{11}{10} = \frac{77}{48} <  \frac{52514}{26325} .$$

We thus may assume that $p_1=3$. By similar logic, we may assume that $p_2 = 5$, and that $p_3$ is one $7$, $11$ or $13.$

Consider the case where $p_3=7$. Then by Lemma \ref{3-5-7-primitive} we have $$h(p_1^{a_1}p_2^{a_2}p_3^{a_3}) \leq \max(h((3^2)(5)(7), h(3)H(5,7)) = \max(\frac{208}{105}, \frac{35}{18}) <\frac{52514}{26325}.$$

Consider now the case where $p_3=11$. Consider the possible values of $a_2$. If $a_2 =1$ (that is $5||N$), then we have
$$h(p_1^{a_1}p_2^{a_2}p_3^{a_3}) \leq H(3)h(5)H(11) = \frac{99}{50}< \frac{52514}{26325}.$$

Now assume $a_2 \geq 2$. This is equivalent to $5^2|N$. Then  by Lemma \ref{3-5-11 primitive lemma}, we have $a_1 \leq 2.$ Thus
$$h(p_1^{a_1}p_2^{a_2}p_3^{a_3}) \leq h(13)H(5)H(11) = \frac{143}{72} < \frac{52514}{26325}.$$ Thus, in all situations where $p_3 =11$ we have the desired inequality. 

Finally, consider the case where $p_3 = 13$.  If $a_2 =1$, then we have $$h(p_1^{a_1}p_2^{a_2}p_3^{a_3}) \leq H(3)h(5)H(13) = \frac{3}{2}\frac{6}{5}\frac{14}{13} = \frac{39}{20} < \frac{52514}{26325}.$$ 

Assume $a_2 \geq 2$. Then by Lemma \ref{3-5-13 primitive}, $a_1 \leq 4$. We will consider four cases depending on the value of $a_1$.

Case {\bf I}: Assume that $a_1=1$. Then 
$$h(p_1^{a_1}p_2^{a_2}p_3^{a_3}) \leq h(3)H(5)H(13) = \frac{4}{3}\frac{5}{4}\frac{13}{12} = \frac{65}{36} < \frac{52514}{26325}.$$

Case {\bf II}: Assume that $a_1 = 2$. Then 
$$h(p_1^{a_1}p_2^{a_2}p_3^{a_3}) \leq h(9)H(5)H(13) = \frac{13}{9}\frac{5}{4}\frac{13}{12} = \frac{845}{432} < \frac{52514}{26325}.$$

Case  {\bf III}: Assume that $a_1 = 3$. If $a_3=1$, then
$$h(p_1^{a_1}p_2^{a_2}p_3^{a_3}) \leq h(27)H(5)h(13) = \frac{40}{27}\frac{5}{4}\frac{14}{13} = \frac{700}{351} < \frac{52514}{26325}.$$

Therefore, we may assume that $a_3 \geq 2$. Thus, by Lemma \ref{3-5-13 primitive}, we have that $a_2 \leq 2$. Thus,

$$h(p_1^{a_1}p_2^{a_2}p_3^{a_3}) \leq h(27)h(25)H(13) = \frac{40}{27}\frac{31}{25}\frac{13}{12} = \frac{806}{405} < \frac{52514}{26325}.$$

Case {\bf IV}: $a_1 =4$. Thus, by Lemma \ref{3-5-13 primitive}, we must have $a_2 \leq 2$, and $a_3 =1$. Then $$h(p_1^{a_1}p_2^{a_2}p_3^{a_3}) \leq h(81)h(25)H(13) = \frac{121}{81}\frac{31}{25}\frac{14}{13} =  \frac{52514}{26325}.$$

Note that it is only in Case IV above that we hit the worst case scenario. 
\end{proof}

Using Theorem \ref{h bound for prim non-deficient with at least four distinct prime factors} along with Equation \ref{52514/26325 explicit b equation} we may use the exact same method of proof to prove the next result.

\begin{proposition} Let $M$ be a primitive non-deficient number. (That is, $M$ is not deficient and it has no non-deficient divisors.) And assume that $M$ has at least four distinct prime divisors. Let $k$ be the number of distinct prime divisors of $M$, and let $p_4$ be the fourth smallest prime factor of $M$. Then we have 
$$k \geq b_{\frac{26325}{26257}}(p_4) +4 .$$ \label{p_4 prim non-deficient bound}
\end{proposition}

As before we have a corresponding corollary.

\begin{corollary}  Let $M$ be a primitive non-deficient number. (That is, $M$ is not deficient and it has no non-deficient divisors.) And assume that $M$ has at least four distinct prime divisors. Let $k$ be the number of distinct prime divisors of $M$, and let $p_4$ be the fourth smallest prime factor of $M$. Then we have 
$$b_{
\frac{26325}{26257}}(p_4) \geq  \frac{p_4^{
\frac{26325}{26257}}}{
\frac{26325}{26257} \log p_4} -  \frac{1.6647
p_4^\frac{26325}{26257} }{\log^2  p_4}  - \frac{0.6603p_4^\frac{26325}{26257}}{ \log^3  p_4} - \frac{p_4}{\log p_4} - \frac{1.2762p_4}{\log^2p_4 } +5. $$
\end{corollary}

Note that the bounds in this section are weak in the following sense: Let $\omega(n)$ be the number of distinct prime divisors of $n$. Given a function $f(n)$ with $\lim_{n \rightarrow \infty} f(n)$, then for any positive integer $i$, the statement ``The $i$th smallest prime divisor of $n$ is at most $f(n)$'' will apply to almost all positive integers $n$ (in the sense that the exceptional set has density zero). Similarly, any such restriction for odd perfect numbers is a weak restriction in the sense discussed in section seven of \cite{Zelinskybig}. \\
\indent It is not too hard to see that all the results in this section which apply to non-deficient numbers in general or apply to primitive non-deficient numbers are asymptotically best possible. Of course, for the results which only concern odd perfect numbers, it is likely that the results are far from best possible, since very likely no odd perfect numbers exist. This is connected to a question raised in  \cite{Zelinskybig}. In particular, the author defined $b_o(p)$ to be the minimum of the number of distinct prime divisors of any odd perfect number with smallest prime divisor $p$ and  set $b_o(p)= \infty$ when there are no odd perfect numbers with smallest prime divisor $p$.  The author asked if it was possible to prove that $b_o(p) > b(p)$, or in the notation of our paper, if $b_o(p) > b_2(p)$. The notation here is less than ideal; note that the small circle in $b_o(p)$ is a lower case o, not a zero. Since we now have a numeric subscript of $b(p)$ in this paper, a slightly better notation for our purposes is as follows: We will write $b_{(O,i)}(p)$ to be the to be the minimum of the number of distinct prime divisors of any odd perfect number with $i$th smallest prime divisor $p$ and  set $b_{(O,i)}(p)= \infty$ when no such odd perfect number exists. The above  mentioned question then becomes whether it is possible to prove that $b_{(O,1)}(p) > b_2(p)$. Following this, we can ask similar questions for other values of $i$. In particular, can we prove any of the following:
\begin{enumerate}
    \item $b_{(O,2)}(p) > b_{\frac{4}{3}}(p)$.
    \item $b_{(O,3)}(p) > b_{\frac{16}{15}}(p)$.
     \item $b_{(O,4)}(p) > b_{\frac{256}{255}}(p)$.
\end{enumerate}

Two other directions which may make sense for investigation are to extend these results for larger prime factors and to extend them to multiperfect numbers. Recall, a number $n$ is said to be $c$-multiperfect if $\sigma(n)=cN$ for some integer $c$. One expects analogous results for multiperfect numbers in terms of $c$. \\

We can also extend various other results  about $a(n)=a_2(n)$ to $a_{\alpha}(n)$ for $\alpha >1$. In \cite{Zelinskybig}, it is proven that $a_2(n)$ is a strictly increasing function. In general, this is false for $a_{\alpha(n)}$. However, we strongly suspect that for any $\alpha > 1$, $a_{\alpha}(n)$ is strictly increasing for sufficiently large $n$. In fact we suspect that

\begin{conjecture} For any $\alpha >1$, and $C>0$, for sufficiently large $n$ one has $a_{\alpha}(n+1) > a_{\alpha}(n) + C.$
\end{conjecture}

In \cite{Zelinskybig}, the second difference of $a(n)$ was studied. The author defined $f(n)=a(n+1)+a(n-1)-2a(n)$. Extending this notation, let us write $f_{\alpha}(n) = a_{\alpha}(n+1)+a_{\alpha}(n-1)-2a_{\alpha}(n)$. Note that $f_{2}(n) = f(n)$.

In that paper, it was noted that $f_2(n)$ is generally positive and that negative values occurred generally at a prime gap. For example,  $f(31)=-5$. This corresponds to the fact that $30$th prime is 113, right before a record setting gap; the next prime is 127.  Motivated by this: we ask the following questions for all $\alpha >1$:

\begin{enumerate}
    \item Are there infinitely many values of $n$ where $f_{\alpha}(n)$ is negative?
    \item Is the set of $n$ where $f_{\alpha}(n)<0$ a subset of those $n$ where $P_n$ is right after a record setting gap?
    \item Are there infinitely many $n$ where $P_n$ occurs at a record setting gap and $f_{\alpha}(n)$ is positive?
    \item Does $f_{\alpha}(n)$ take on any integer value? In particular, is $f(n)$ ever zero?
    \item Does $a_{\alpha}(n+1)-a_{\alpha}(n)$ take on every positive integer value? 
\end{enumerate}

We suspect that there are values of $\alpha$ where questions 4 and 5 have an answer of no. We are substantially more uncertain about the other questions. 

\section{Sums of divisors of numbers near $n!$}

It is well known that for any $n>1$, any number of the form $n! \pm 1$ is deficient. We reproduce a proof here. 

\begin{lemma} For any prime $p>2$, we have $b_2(p) \geq p$.
\label{Servais lemma}
\end{lemma}
\begin{proof} Assume that $b_2(p)=m$. Then since $\frac{x}{x-1}$ is a decreasing function for positive $x$ and the $i$th prime after $p$ is at least $p+i$, we must have
$$2 <  \left(\frac{p}{p-1}\right)\left(\frac{p+1}{p} \right)\left(\frac{p+2}{p+1}\right) \cdots \left(\frac{p+b(p)-1}{p+b(p)-2}\right) = \frac{p+b(p)-1}{p-1}.  $$
Thus $$2p-2 < p+b(p) -1,$$
and so $b_2(p) > p-1$, and so $b_2(p) \geq p$.
\end{proof}
Variants of Lemma \ref{Servais lemma} are very old and to date back to Servais \cite{Servais} who used it to prove that an odd perfect number with smallest prime divisor $p$ must have at least $p$ distinct prime factors.  It immediately follows from this Lemma that any non-deficient number with smallest prime factor $p$ must have at least $p$ distinct prime divisors. 

To prove from that that any number of the form $n! \pm 1$ must be deficient, note that any such number must have smallest prime factor at least $n+1$. If $n! \pm 1$ were not-deficient it would need to have at least $n+1$ distinct prime factors, but then we would have $(n+1)^{n+1} \leq n! \pm 1$, which is absurd. 

In this section, we will use the same methods as earlier in this paper to prove the related result that in fact
$$\lim_{n \rightarrow \infty} h(n! \pm 1) = 1.$$

In fact, we will in fact prove a stronger claim.

\begin{theorem} Let $a$, $b$ and $c$ be positive integers. Then $$\lim_{n \rightarrow \infty} h(\frac{a}{b}n! \pm c) = h(c).$$  
\label{big factorial theorem}
\end{theorem}
Note that simply using the method of Servais's Lemma is insufficient to prove Theorem \ref{big factorial theorem}. Using only that method, one would just get that a number $n$ with $h(n) \geq \alpha$ with smallest prime factor $p$ has at least about $(\alpha-1)p$ distinct prime factors. So we really do need to use Mertens' theorem.

To prove Theorem \ref{big factorial theorem}, we will prove a slightly stronger claim.

\begin{theorem} $$\lim_{n \rightarrow \infty} H\left(\frac{a}{b}n! \pm 1\right) = 1.$$ \label{small case factorial limit theorem}
\end{theorem}
\begin{proof} We will prove the result with $\frac{a}{b}n! - 1$. The proof is nearly identical for $\frac{a}{b}n! + 1$. \\ 

Fix positive integers $a$ and $b$. Fix an $\alpha >1$, and assume that there is an increasing sequence $n_1, n_2 \cdots $ such that for all $k$, $\frac{a}{b}n!$ is an integer and $H(\frac{a}{b}n_k! \pm 1) \geq \alpha$. For sufficiently $n_k$, the smallest prime divisor of $\frac{a}{b}n_k! - 1$ is at least $n+1$. According to Theorem \ref{asymptotic for b_alpha} for any $\epsilon>0$, for sufficiently large $n_k$ we have
that $\frac{a}{b}n_k! \pm 1$ has at least $(1-\epsilon)\frac{(n_k+1)^{\alpha}}{\alpha \log (n_k+1)}$ distinct prime factors. Choose $\epsilon=1/2$. Then for sufficiently large $n_k$, we must have
$$(n_k+1)^{\frac{1}{2}\frac{(n_k+1)^{\alpha}}{\alpha \log (n_k+1)}} <  \frac{a}{b}n_k! -1 < a {n_k}^{n_k}.$$ This implies that
$$(n_k)^{\frac{1}{2}\frac{(n_k)^{\alpha}}{\alpha \log (n_k+1)}} <  \frac{a}{b}n_k! -1 < a {n_k}^{n_k},$$

and hence $$\frac{1}{2}\frac{(n_k)^{\alpha}}{\alpha \log (n_k+1)} < n_k.$$

However, this inequality cannot be true for any sufficiently large $n_k$ since $\alpha >1$. Thus, there are only finitely many possible values for $n_k$ and the result is proved.
\end{proof}

With a little work one can use essentially the exact same method as above to prove the following slightly more general result:

\begin{theorem}
Let $g(n)$ be an increasing function satisfying the following conditions:
\begin{enumerate}
    \item If $n$ is sufficiently large, then for any prime $p$ with $p \leq n$, $(p,g(n))=1$.
    \item  There is a constant $C$ such that $g(n) < n^{Cn}$ for all sufficiently large $n$. 
\end{enumerate}
Then $$\lim_{n \rightarrow \infty} H(g(n)) = 1.$$
\end{theorem}

Four functions which seem natural  to investigate in the context of Theorem \ref{small case factorial limit theorem}, are $A(n)=h(n!+1)$, $B(n)=h(n!-1)$, $C(n)=H(n!+1)$, and $D(n)=H(n!+1)$. We will focus on $A(n)$, although the same remarks apply essentially to the other three.

One might hope that $A(n)$ would be a decreasing function given our earlier limit and the fact that $A(1) = \frac{3}{2}$, and that $A(n) >1$ for all $n>1$. However, this is not the case. $A(3)=h(7)=\frac{8}{7} = 1.1428 \cdots $, and $A(4) = h(25) = \frac{31}{25} = 1.24$.  It is easy to find similar values for $B(n)$, $C(n)$, and $D(n)$.

We suspect that this non-decreasing behavior is not simply an artifact of small numbers. In particular we have the following conjecture.

\begin{conjecture} There are arbitrarily large $n$ such that $A(n)< A(n+1)$. \label{A(n) bump conjecture}
\end{conjecture}

We have similar conjectures for the other three function.

\begin{conjecture} There are arbitrarily large $n$ such that $B(n)< B(n+1)$. \label{B(n) bump conjecture}
\end{conjecture}

\begin{conjecture} There are arbitrarily large $n$ such that $C(n)< C(n+1)$. \label{C(n) bump conjecture}
\end{conjecture}

\begin{conjecture} There are arbitrarily large $n$ such that $D(n)< D(n+1)$. \label{D(n) bump conjecture}
\end{conjecture}

Note that standard conjectures about primes imply these conjectures. For example, whenever $n!+1$ is prime and $(n+1)!+1$ is composite, one will have that $A(n) < A(n+1)$ and $C(n) < C(n+1)$. Since it is believed that there are infinitely many prime values of $n!+1$, and since there are infinitely many values of $n!+1$ which are composite (by Wilson's Theorem), we should strongly believe Conjecture \ref{A(n) bump conjecture}. One hopes that these conjectures may be substantially easier to prove than statements like there being infinitely many primes of the form $n!+1$.

One obvious question is if there are infinitely many $n$ where $n!+1$ is composite but $A(n) < A(n+1)$. This question seems difficult but we tentatively conjecture that it is also yes. To see why this seems likely, note that if $n!+1$ is a semiprime with both prime factors near the square root, and $n!+1$ has a small prime divisor substantially smaller than its square root, we will have $A(n) < A(n+1)$.

Let $S_A$ be the set of positive integers such that $A(n) < A(n+1)$. Define $S_B$, $S_C$ and $S_D$ similarly for $B(n)$, $C(n)$, and $D(n)$. What is the density of these sets? The obvious conjectural natural densities for $S_A$ is 1/2 given the apparent slow rate which $A(n)$ tends towards 1. We do not see any way of even proving that this density exists. Curiously, we do not at this time see a way of even ruling out that $S_A$ has density 1. If that were the case, it would mean that $A(n) < A(n+1)$ on almost all $n$ but that $A(n)$ takes large drops at isolated $n$ which drive it downwards overall. This seems like it should be obviously false but we do not see any easy proof. \\

The following conjecture seems likely.

\begin{conjecture} $$\lim_{n \rightarrow \infty} H\left(\left(n!\right)^n+1\right)=1.$$ \label{Generous conjecture}
\end{conjecture}

Unfortunately, Mertens' Theorem by itself does not seem to be strong enough to prove Conjecture \ref{Generous conjecture}. It would likely require some result of the sort that the number of distinct prime factors of $(n!)^n +1$ must be much smaller than $\log_n (n!)$.

We use this paper as an opportunity to ask similar questions about $h(f(n))$ for other somewhat well-behaved fast growing functions $f(n)$. Let us consider the functions $J(n) = h(2^n-1)$ and $K(n) = h(2^n+1)$, $L(n) = H(2^n-1)$, and $M(n) = H(2^n+1)$. Unlike with our earlier quadruplet, the behavior in the $-1$ case here is substantially different than the $+1$ case. It is a straightforward consequence from Fermat's Little Theorem, and the fact that the sum of the reciprocals of the primes diverges, that $J(n)$ can be arbitrarily large. 

In contrast, it is less obvious that $K(n)$ can be arbitrarily large.
\begin{theorem} $K(n)$ can be arbitrarily large.
\end{theorem}
\begin{proof} Consider a prime $p \equiv 3$ (mod 8). Note that $2$ is a quadratic non-reside mod $p$. Thus, by Gauss's Lemma, $2^{\frac{p-1}{2}} \equiv -1$ (mod $p$). So, whenever $n$ is odd, and $\frac{p-1}{2}|n$, we have $p|2^n+1$. By the strong form of Dirichlet's theorem,  the sum of the reciprocals of the primes which are 3 mod 8 diverges. Thus, to make $h(2^n+1)$ arbitrarily large, we just need to set $n$ to be the odd part of $k!$ for large $k$. 
\end{proof}

One natural question in this context is to look at record setting values of $K(n)$. This becomes an analog of the superabundant numbers restricted to just looking at one more than a power of 2. More generally, given an increasing function $f(n)$, we will define a number to be a {\emph{$f$-superabundant}} if $h(f(n))> h(f(i))$ for any $i<n$. Note that the usual superabundant numbers are just the $f$-superabundant numbers where $f(x)$ is the identity function. 

In our context then, we are interested in the $K$-superabundant numbers. We will write the $n$th $K$-superabundant number as $a_n$. The first few are 
$1$, $3$,\, $5$, $9$, $15,$ $45$, $135,$ $315,$ $945 \cdots$. Note that the first few $a_i$ are odd. It seems likely that all the $a_i$ are odd because if $n$ is even, then $2^n+1$ is not divisible by 3, and that makes making $2^n+1$ large substantially more difficult. Proving that they are all odd seems difficult. A related open question is whether for any odd prime $p$, if $p|a_n$ for all arbitrarily large $n$? The answer seems very likely to be yes.

Note that the $J$-superabundant numbers behave very differently. They appear to be always even except the first term, and seem to grow much slower. The first few are 1, 2, 4, 6, 8, 12, 24, 36, 60, 120, and 180. (Here we are using the convention that $\sigma(0)=0$.)  If we set $t_n$ as the $n$th $J$-superabundant number, we can similarly ask if it is true that $p|t_n$ for all arbitrarily large $n$? Again, the answer seems likely to be yes.

\end{document}